\documentclass[final,12pt,3p]{elsarticle}

\usepackage{graphicx}
\usepackage{amssymb}
\usepackage{amsthm}
\usepackage{amsmath}
\usepackage{algorithm}
\usepackage{algorithmic}
\usepackage{longtable} 
\usepackage{tikz}
\usepackage{lineno}
\usepackage{verbatim} 
\usepackage{hyperref}
\usepackage[dvipsnames]{xcolor}
\usepackage{caption}
\usetikzlibrary{shapes}
\usepackage{subcaption}
\usepackage{mathrsfs} 
\usepackage{array}
\usepackage{multirow}
\usepackage{booktabs} 
\usepackage{adjustbox}
\usepackage[title]{appendix}

\biboptions{sort&compress}
\hypersetup{
colorlinks=true,
linkcolor=red,
filecolor=green,
urlcolor=green,
anchorcolor=green,
citecolor=cyan,
pdftitle={Overleaf Example},
pdfpagemode=FullScreen,
}
\theoremstyle{plain}
\newtheorem{theorem}{Theorem}
\newtheorem{proposition}[theorem]{Proposition}

\theoremstyle{definition}

\theoremstyle{remark}
\newtheorem{remark}{Remark}

\DeclareMathOperator{\diver}{div}

\newcommand{\good}[1]{\textcolor{OliveGreen}{\textbf{#1}}}
\newcommand{\bad}[1]{\textcolor{BrickRed}{#1}}
\begin{document}

\begin{frontmatter}

\title{A Nonstandard Finite Difference Scheme for 
a Nonlinear Parabolic Equation with p-Laplacian-Type Diffusion}

\author[BUW,MAIS]{Achraf Zinihi}
\ead{a.zinihi@edu.umi.ac.ma} 

\author[BUW]{Matthias Ehrhardt\corref{Corr}}
\cortext[Corr]{Corresponding author}
\ead{ehrhardt@uni-wuppertal.de}

\author[MAIS]{Moulay Rchid Sidi Ammi}
\ead{rachidsidiammi@yahoo.fr}

\address[BUW]{University of Wuppertal, Applied and Computational Mathematics,\\
Gaußstrasse 20, 42119 Wuppertal, Germany}

\address[MAIS]{Department of Mathematics, AMNEA Group, Faculty of Sciences and Techniques,\\
Moulay Ismail University of Meknes, Errachidia 52000, Morocco}


\begin{abstract}
We propose and analyze a nonstandard finite difference (NSFD) scheme for nonlinear parabolic equations involving a $p$-Laplacian-type diffusion operator in one- and two-dimensional spatial domains. 
Following Mickens’ design principles, the proposed discretization employs a nonlinear denominator function $\phi(\cdot)$ together with a nonlocal approximation of the nonlinear diffusion term $\Delta_p$, yielding a structure-preserving discrete model. 
The scheme is designed to retain key qualitative properties of the continuous problem, including positivity, boundedness, and stability, which may be lost by standard finite difference methods (FDMs). 
We establish the well-posedness of the continuous model, derive the NSFD scheme, and investigate its consistency, convergence, and local truncation error. 
Numerical experiments confirm the theoretical results and demonstrate that, unlike the standard explicit FDM, the proposed NSFD scheme avoids spurious oscillations and nonphysical negative solutions even for relatively large time-step sizes.
\end{abstract}


\journal{Applied Mathematics Letters}


\end{frontmatter}


\section{Introduction}\label{S1}
Numerical simulations have become an indispensable tool for investigating nonlinear parabolic partial differential equations. 
In many applications, analytical solutions are unavailable, making numerical methods the primary means for exploring the qualitative and quantitative behavior of these systems. 
This challenge is particularly pronounced for nonlinear parabolic partial differential equations posed in one- or two-dimensional spatial domains, where nonlinear diffusion mechanisms, intricate spatial interactions, and geometric effects significantly increase computational complexity. 

Let $\Omega \subset \mathbb{R}^n$ be a bounded domain with smooth boundary $\partial\Omega$, where $n = 1$ or $2$, and let $f$ be a sufficiently smooth function.
We define $\Phi_p(y) = y^{p-2}$, $p > 1$, and $\lambda > 0$.
Nonlinear parabolic equations of the form
\begin{equation}\label{E1.1}
\left\{\begin{aligned}
& \frac{\partial u}{\partial t} = \lambda\diver(\Phi_p(|\nabla u|)\nabla u) + f(u), \quad \text { in } \mathcal{U} = [0, T] \times \Omega,\\ 
& \nabla u \cdot \vec{n}=0, \quad  \text { on } \Sigma = (0, T) \times \partial\Omega, \\ 
& u(0, \cdot) = u_0(\cdot), \quad  \text { in } \Omega,
\end{aligned}\right.
\end{equation}
arise in a broad range of applications including image processing \cite{Atlas2014}, mathematical epidemiology \cite{Zinihi2025FDE}, 
porous medium flows \cite{Vazquez2006}, and nonlinear heat conduction \cite{Lindqvist2019}.
The case $p = 2$ recovers the classical linear diffusion equation, while $p \ne2$ introduces strong nonlinearity through 
the so-called $p$-Laplacian operator $\Delta_p u = \diver(|\nabla u|^{p-2}\nabla u)$. 
A fundamental requirement in many of these applications is that the numerical solution remain non-negative and uniformly bounded, 
reflecting the physical or biological meaning of the quantity $u$ (e.g., a concentration or density). 
Standard explicit \textit{finite difference methods} (FDMs) generally fail to preserve these properties unless stringent step-size restrictions are imposed, 
and they may produce spurious oscillations or nonphysical negative values for moderate time steps.

To address these shortcomings, we propose a \textit{nonstandard finite difference} (NSFD) scheme in the spirit of \cite{Mickens1999, Mickens2006}. 
NSFD methods replace the standard discrete derivative $(u^{m+1}-u^m)k^{-1}$ by a generalized counterpart involving a denominator function $\phi(k)$, 
satisfying $\phi(k) = k + \mathcal{O}(k^2)$, and treat nonlinear terms through nonlocal representations on the computational grid. 
This approach has proven effective for positivity preservation in actuarial-insurance models \cite{Zinihi2026Actuarial}, 
convection–diffusion equations \cite{Ehrhardt2013}, epidemic reaction-diffusion systems \cite{Zinihi2025NSFD}, 
machine learning-based epidemic modeling frameworks \cite{Zinihi2026CA}, among others.

The construction of denominator functions satisfying positivity constraints was systematically investigated by Mickens \cite{Mickens1999, Mickens2006}. 
Zinihi, Ehrhardt, et al. \cite{Zinihi2026Actuarial} proposed an actuarial framework based on an epidemic model and employed an NSFD scheme to numerically investigate epidemic-adjusted insurance quantities and disease-driven mortality effects.
Ehrhardt and Mickens \cite{Ehrhardt2013} extended the methodology to convection-diffusion equations using the subequation approach. 
More recently, Zinihi et al. \cite{Zinihi2025NSFD} developed a positivity-preserving NSFD scheme for a reaction-diffusion epidemic model.
Zinihi and Ehrhardt \cite{Zinihi2026CA} employed a structure-preserving NSFD scheme to generate synthetic data that ensures positivity, boundedness, and numerical stability of the computed solutions.

The standard FDM for~\eqref{E1.1}, which motivates the present work, considers a uniform discretization of the spatial domain $\Omega$. 
Let $(x_i,y_j)$, $i=0,\ldots,N_x$, $j=0,\ldots,N_y$, denote the grid points with uniform mesh size $h_x = h_y = h$. 
The temporal interval is discretized using a time step $k$, and the approximation of $u(t_m, x_i, y_j)$ at time $t_m = m k$ is denoted by $u_{ij}^m$.
The nonlinear coefficient $\Phi_p$ is evaluated using central-difference approximations of the gradient. 
For convenience, we introduce the centered and forward difference quotients
\begin{equation*}
   D_{h,x}^0 u_{ij}^m  = \frac{u_{i+1,j}^m-u_{i-1,j}^m}{2h}, 
   \quad 
   D_{h,x}^+ u_{ij}^m = \frac{u_{i+1,j}^m-u_{ij}^m}{h},
\end{equation*}
and define
\begin{equation*}
\Delta_{h,x} u_{i+1/2,j}^m = D_{h/2,x}^0 \Bigl(u_{i+1/2,j}^m\, 
\Phi_{p,,i+1/2,j}^m
\Bigr), \quad \text{ with } \;\Phi_{p,,ij}^m = \Phi_p\Bigl(\sqrt{ (D_{h,x}^0 u_{ij}^m)^2 + (D_{h,y}^0 u_{ij}^m)^2}\Bigr),
\end{equation*}
and the averages $\Phi_{p,,i+1/2,j}^m:=(\Phi_{p,,i+1,j}^m+\Phi_{p,,ij}^m)/2$.
The quantities 
$D_{h,y}^0 u_{ij}^m$, $D_{h,y}^+ u_{ij}^m$, and $\Delta_{h,y}u_{ij}^m$ are defined analogously.
The resulting explicit FDM discretization of~\eqref{E1.1} is given by
\begin{equation*}
    D^+_k u_{ij}^m = \lambda \Bigl(D_{h/2,x}^0 \Delta_{h,x} u_{ij}^m + D_{h/2,y}^0 \Delta_{h,y} u_{ij}^m \Bigr)+ f(u_{ij}^m),
\end{equation*}
i.e.
\begin{equation}\label{E1.2}
\begin{split}
   u_{ij}^{m+1} &= u_{ij}^m + \lambda\frac{k}{h} \Big(
   \Delta_{h,x}u_{i+1/2,j}^m - \Delta_{h,x}u_{i-1/2,j}^m  + \Delta_{h,y}u_{i,j+1/2}^m - \Delta_{h,y} u_{i,j-1/2}^m \Big) + k f(u_{ij}^m),\\
   &= u_{ij}^m + \lambda r \Big(\bigl(u_{i+1,j}^m\,\Phi_{p,,i+1,j}^m- u_{ij}^m\,\Phi_{p,,ij}^m \bigr) 
    - \bigl(u_{ij}^m\, \Phi_{p,,ij}^m- u_{i-1,j}^m\,\Phi_{p,,i-1,j}^m \bigr)\\
    &\qquad+ \bigl(u_{i,j+1}^m\,\Phi_{p,,i,j+1}^m- u_{ij}^m\,\Phi_{p,,ij}^m \bigr)
    - \bigl(u_{ij}^m\, \Phi_{p,,ij}^m- u_{i,j-1}^m\,\Phi_{p,,i,j-1}^m \bigr) \Big) + k f(u_{ij}^m),
\end{split}
\end{equation}
with the parabolic mesh ratio $r=k/h^2$.
The homogeneous Neumann boundary conditions are discretized using second-order central finite differences as
\begin{equation}\label{E1.3}
\begin{aligned}
  u_{-1,-1}^{m} &= u_{1,1}^{m}, & u_{N_x+1,-1}^{m} &= u_{N_x-1,1}^{m}, \quad u_{-1,N_y+1}^{m} = u_{1,N_y-1}^{m}, \quad u_{N_x+1,N_y+1}^{m} = u_{N_x-1,N_y-1}^{m},\\
   u_{-1,j}^{m} &= u_{1,j}^{m}, & u_{N_x+1,j}^{m} &= u_{N_x-1,j}^{m}, \quad \forall j=0,\ldots,N_y,\\
  u_{i,-1}^{m} &= u_{i,1}^{m}, & u_{i,N_y+1}^{m} &= u_{i,N_y-1}^{m}, \quad \forall i=0,\ldots,N_x.
\end{aligned}
\end{equation}
Although the scheme~\eqref{E1.2}--\eqref{E1.3} is consistent with the continuous problem, it does not necessarily preserve positivity. 
In particular, for sufficiently large time steps, the coefficient associated with the central node $u_{ij}^m$ may become negative, potentially leading to nonphysical oscillations and violations of the discrete maximum principle. 
This observation motivates the development of an NSFD scheme capable of preserving the qualitative properties of the underlying continuous model.

The goal of this paper is to construct an NSFD scheme that 
(i) is explicit and easy to implement, 
(ii) unconditionally preserves positivity and boundedness, 
(iii) is consistent with the continuous PDE, and 
(iv) outperforms (FDM) on coarse time grids. 
The paper is organized as follows. 
Section~\ref{S2} presents the NSFD scheme. 
Section~\ref{S3} establishes its theoretical properties. 
Section~\ref{S4} provides numerical experiments.

\section{Nonstandard Finite Difference Scheme}\label{S2}

This section outlines the fundamental principles of NSFD schemes. These methods are designed to preserve the key qualitative properties of the underlying differential equations, such as positivity and conservation laws. 
Consequently, the numerical solutions remain bounded and stable.
Under standard regularity assumptions on $f$ and $u_0$, problem~\eqref{E1.1} admits a unique nonnegative weak solution
\begin{equation*}
    u\in C([0,T],L^2(\Omega)) \cap L^p(0,T;W_0^{1,p}(\Omega)),
\end{equation*}
see \cite[pp.~5--7]{Geredeli2014}. 
For further details, we refer the reader to \cite{Brezis2011}.

\subsection{The Denominator Function}\label{S2.1}

We briefly review the main principles of NSFD schemes, as introduced by Mickens \cite{Mickens1999}; see also \cite[p.~6]{Zinihi2025NSFD}. 
A finite difference scheme is an NSFD scheme if it satisfies at least one of Mickens's nonstandard discretization rules.
First, the order of the discrete derivative must coincide with that of the corresponding continuous derivative. First-order derivatives are commonly approximated by
\begin{equation*}
    \frac{d u}{dt}\Big|_{t=t_m} \approx \frac{u^{m+1}-u^m}{\phi(k)},
\end{equation*}
where the denominator function $\phi(k)>0$ satisfies $\phi(k)=k+\mathcal{O}(k^2)$, thereby ensuring consistency while improving the qualitative behavior of the numerical solution.

Second, nonlinear terms are often discretized using nonlocal representations involving values at different time levels. For instance, $u^2(t_m)\approx u^m u^{m+1}$ or $u^3(t_m)\approx (u^m)^2u^{m+1}$.
Finally, the discrete model should preserve the essential qualitative properties of the continuous problem, such as positivity, boundedness, equilibrium points, and their stability.

NSFD schemes approximate the time derivative by a generalized difference quotient
\begin{equation*}
\frac{d u}{dt}\Big|_{t=t_m} \approx \frac{u^{m+1}-\psi(k)u^m}{\phi(k)},
\end{equation*}
where $\psi(k)=1+\mathcal{O}(k)$. A typical choice is $\phi(k)=(e^{Lk}-1)L^{-1}$, where $L>0$ is a Lipschitz constant of the right hand side of~\eqref{E1.1}.

\subsection{Nonlocal Reaction Discretization}\label{S2.2}
When the right-hand side of~\eqref{E1.1} is linear, its discretization coincides with the classical FDM approximation. In the nonlinear case, since the function $f$ is not specified, several NSFD discretizations can be constructed; see, for example, \cite[pp.~11--13]{Zinihi2026Actuarial} and \cite[p.~6]{Zinihi2025NSFD}. 
In most applications the source function $f$ is positive and thus one can simply
use $f(u(x_j,y_j,t_m))\approx f(u_{ij}^m)$.
However, to ensure positivity in the general case, the arbitrary function $f$ is decomposed into its positive and negative parts at $u_{ij}^{m}$, with $f^+(u_{ij}^{m})$ 
is treated explicitly and the negative part is approximated nonlocally as $f^-(u_{ij}^{m+1})\approx f^-(u_{ij}^m)\,u_{ij}^{m+1}/u_{ij}^m$. 
Consequently, the proposed NSFD scheme for~\eqref{E1.1} is given by
\begin{equation}\label{E2.1}
\frac{u_{ij}^{m+1}-u_{ij}^{m}}{\phi(k)} = \lambda \mathcal{D}_h^{NS} u_{ij}^m + f^+(u_{ij}^m) - f^-(u_{ij}^{m+1}),
\end{equation}
where $\mathcal{D}_h^{NS}$ denotes the nonstandard discretization of the p-Laplacian$\Delta_p$.

Additionally, the discretization of the nonlinear diffusion operator must be modified according to Mickens's NSFD principles. 
To this end, we introduce a nonlocal (two time levels) approximation to the discrete fluxes by evaluating the central node at the future time level while retaining the neighboring nodes at the current time level. 
Specifically, the standard forward difference quotients $D_{h,x}^+$ and $D_{h,y}^+$ in \eqref{E1.2} are replaced by their two time levels counterparts, denoted by $D_{h,x}^{+,NS}$ and $D_{h,y}^{+,NS}$, which are obtained by substituting $-u_{ij}^m$ with $-u_{ij}^{m+1}$. 
Thus, this nonstandard, 'skew' spatial discretization reads 
\begin{equation}\label{E2.2}
\begin{split}
\mathcal{D}_h^{NS} u_{ij}^m =\,& \frac{1}{h} \Big(
\Delta_{h,x}^{NS} u_{i+1/2,j}^m - \Delta_{h,x}^{NS} u_{i-1/2,j}^m 
+ \Delta_{h,y}^{NS} u_{i,j+1/2}^m - \Delta_{h,y}^{NS} u_{i,j-1/2}^m \Big),\\
=\,& \frac{1}{h^2} \Big(
\bigl(u_{i+1,j}^m\,\Phi_{p,,i+1,j}^m - u_{ij}^{m+1}\,\Phi_{p,,ij}^m \bigr) 
- \bigl(u_{ij}^{m+1}\,\Phi_{p,,ij}^m - u_{i-1,j}^m\,\Phi_{p,,i-1,j}^m \bigr)\\
&\,+ \bigl(u_{i,j+1}^m\,\Phi_{p,,i,j+1}^m - u_{ij}^{m+1}\,\Phi_{p,,ij}^m \bigr)
- \bigl(u_{ij}^{m+1}\,\Phi_{p,,ij}^m - u_{i,j-1}^m\,\Phi_{p,,i,j-1}^m \bigr) 
\Big)\\
=\,& \frac{1}{h^2} \Big(
u_{i+1,j}^m\,\Phi_{p,,i+1,j}^m + u_{i,j+1}^m\,\Phi_{p,,i,j+1}^m - 4 u_{ij}^{m+1}\,\Phi_{p,,ij}^m 
 + u_{i-1,j}^m\,\Phi_{p,,i-1,j}^m 
+ u_{i,j-1}^m\,\Phi_{p,,i,j-1}^m \Big).
\end{split}
\end{equation}
The two-level representation~\eqref{E2.2} preserves the dissipative character of the $p$-Laplacian opera\-tor and yields a discrete diffusion term that is a convex combination of neighboring values. 
Consequently, the resulting NSFD scheme inherits important qualitative properties of the continuous problem, such as positivity preservation and a discrete maximum principle. 
Furthermore, these properties are obtained without the restrictive time-step constraints usually necessary for the standard finite difference approximation. However, as a drawback, the order in time is reduced from 2 to 1, compared to a standard Crank-Nicolson FDM.
A recent paper \cite{hoang2026generalized} provides strategies for overcoming this order reduction.

\subsection{Interior and Boundary Schemes}\label{S2.3}

Using the two-level flux approximation introduced above, the proposed NSFD scheme~\eqref{E2.1} can be written as
\begin{equation}\label{E2.3}
\begin{aligned}
u_{ij}^{m+1} = \frac{1}{1 + q_{-,,ij}^m + 4\lambda r_\phi \Phi_{p,,ij}^m} \Big[
&u_{ij}^m + \lambda r_\phi \Big(u_{i+1,j}^m\,\Phi_{p,,i+1,j}^m 
+ u_{i-1,j}^m\,\Phi_{p,,i-1,j}^m\\
&+ u_{i,j+1}^m\,\Phi_{p,,i,j+1}^m 
+ u_{i,j-1}^m\,\Phi_{p,,i,j-1}^m\Big) + \phi(k) f^+(u_{ij}^m) 
\Big],
\end{aligned}
\end{equation}
subject to the homogeneous Neumann boundary conditions~\eqref{E1.3}, where $r_\phi=\phi(k)/h^2$ denotes the generalized (updated) parabolic mesh ratio and $q_{-,,ij}^m = \phi(k) f^-(u_{ij}^m) / u_{ij}^m$. 

\begin{remark}\label{R2}
In one space dimension, the NSFD scheme~\eqref{E2.3} simplifies to
\begin{equation}\label{E2.4}
   u_{i}^{m+1} = \frac{u_{i}^m + \lambda r_\phi \big(u_{i+1}^m\,\Phi_{p,,i+1}^m 
   + u_{i-1}^m\,\Phi_{p,,i-1}^m\big) + \phi(k) f^+(u_{i}^m)}{1 + q_{-,,i}^m + 2\lambda r_\phi\,\Phi_{p,,i}^m}. 
\end{equation}
\end{remark}






\section{Theoretical Analysis of the proposed Scheme}\label{S3}
In this section, we establish the main qualitative properties of the proposed NSFD scheme. 
Throughout, we assume that the exact solution and its discrete approximations remain positive, and that the reaction term satisfies the one-sided growth condition
\begin{equation*}
   f(s)\le Ls, \quad \forall s\ge 0,
\end{equation*}
for some constant $L>0$. 
This assumption is natural in many reaction--diffusion settings and will be used in the proof of the boundedness result in Theorem~\ref{T3}.

First, we show that \eqref{E2.3} preserves positivity.
Recall that the denominator function $\phi$ satisfies $\phi(k) > 0$ for all $k > 0$, and that $\Phi_p(s) = s^{p-2} > 0$ for all $s > 0$ and $p > 1$.
Thus, the mesh ratio $r_\phi = \phi(k)/h^2$ is strictly positive.
Furthermore, the splitting $f = f^+ - f^-$ ensures 
\begin{equation*}
   f^+(u) \ge0 \quad\text{ and } \quad f^-(u) \ge0 \quad  \text{ for all } \;  u,
\end{equation*} 
so that both the reaction contribution $\phi(k)f^+(u_{ij}^m) \ge0$ 
and the penalty term $q_{-,ij}^m = \phi(k)f^-(u_{ij}^m)/u_{ij}^m \ge0$ are nonnegative.
Assuming $u_{ij}^m > 0$ at all grid points, the neighbor term
\begin{equation*}
\lambda r_\phi\, u_{i\pm 1,j}^m\,\Phi_{p,i\pm 1,j}^m \quad \text{ and } \quad 
\lambda r_\phi\, u_{i,j\pm 1}^m\,\Phi_{p,i,j\pm 1}^m
\end{equation*}
are strictly positive, making the entire numerator of~\eqref{E2.3} strictly positive.
The denominator satisfies 
\begin{equation*}
    1 + q_{-,ij}^m + 4\lambda r_\phi\,\Phi_{p,ij}^m \ge1 > 0,
\end{equation*}
so the ratio $u_{ij}^{m+1}$ is strictly positive, as stated in the following theorem.


\begin{theorem}[Positivity Preservation]\label{T1}
Assume that we have $u_{ij}^m>0$ for all grid points $(i,j)$ at time level $m$. 
Then, the NSFD scheme~\eqref{E2.3} satisfies $u_{ij}^{m+1} > 0$, for all grid points $(i,j)$.
\end{theorem}

Next we investigate whether~\eqref{E2.3} preserves the steady states of the continuous problem~\eqref{E1.1}. 
The following theorem summarizes the result.

\begin{theorem}[Constant Equilibrium Preservation]\label{T2}
Let $u^*$ be a constant equilibrium of~\eqref{E1.1}, i.e., $f(u^*)=0$.
Then the NSFD scheme~\eqref{E2.3} preserves this equilibrium exactly. 
More precisely, if $u_{ij}^m = u^*$ for all grid points $(i,j)$, then $u_{ij}^{m+1} = u^*$.
\end{theorem}
\begin{proof}
Let $u^*$ be a constant equilibrium of \eqref{E1.1}.
Suppose that $u_{ij}^m = u^*$ for all grid points $(i,j)$.
Since the solution is spatially uniform, all discrete differences vanish:
\begin{equation*}
    D_{h,x}^0 u_{ij}^m = D_{h,y}^0 u_{ij}^m = 0,
\end{equation*}
and therefore 
   $\Phi_{p,,ij}^m = \Phi_p(0) = 0$,
so all diffusion terms drop out.
Since $f(u^*) = 0$, the decomposition $f = f^+ - f^-$ yields $f^+(u^*) = f^-(u^*)$.
Denoting this common value by $f^*(u^*) \ge0$ and setting $q_-^* = \phi(k) f^-(u^*)/u^*$, substituting $u_{ij}^m = u^*$ into \eqref{E2.3} yields
\begin{equation*}
u_{ij}^{m+1} = \frac{u^* + \phi(k)\,f^+(u^*)}{1 + q_-^*}
= \frac{u^*\bigl(1 + \phi(k)\,f^-(u^*)/u^*\bigr)}{1 + \phi(k)\,f^-(u^*)/u^*}
= u^*.
\end{equation*}
Therefore, the constant equilibrium $u^*$ is preserved exactly by the scheme~\eqref{E2.3}.
\end{proof}

We now establish an a priori $\ell^\infty$ bound for the discrete solution generated by \eqref{E2.3}. 
The key feature of the scheme is that the negative part of the reaction term is incorporated into the denominator, while the positive part remains in the numerator. 

\begin{theorem}[Boundedness]\label{T3}
Assume that the initial data $u_{ij}^0$ are bounded. 
Then the solution of the NSFD scheme~\eqref{E2.3} remains bounded in the discrete $\ell^\infty$-norm.
\end{theorem}
\begin{proof}
According to Theorem~\ref{T1}, the numerical solution remains positive for all time levels. 
Let
\begin{equation*}
U^m := \max_{i,j} u_{ij}^m \quad \text{ and } \quad \Phi_{\max}^m := \max_{i,j} \Phi_{p,ij}^m.
\end{equation*}
Since $u_{ij}^m \le U^m$ for all neighboring nodes and $\Phi_{p,ij}^m \le \Phi_{\max}^m$, we have
\begin{equation*}
u_{i+1,j}^m \Phi_{p,i+1,j}^m + u_{i-1,j}^m \Phi_{p,i-1,j}^m + u_{i,j+1}^m \Phi_{p,i,j+1}^m + u_{i,j-1}^m \Phi_{p,i,j-1}^m \le 4U^m \Phi_{\max}^m.
\end{equation*}
Since the denominator satisfies $1 + q_{-,ij}^m + 4\lambda r_\phi \Phi_{p,ij}^m \ge 1$,
it follows that
\begin{equation*}
   U^{m+1} \le \bigl(1 + 4\lambda r_\phi \Phi_{\max}^m\bigr)\, U^m + \phi(k) f^+(U^m).
\end{equation*}
Using Theorem~\ref{T1}, we have $f^+(u_{ij}^m) = \max(f(u_{ij}^m), 0) \le \max(L u_{ij}^m, 0) = L u_{ij}^m$. Thus
\begin{equation}\label{E3.1}
U^{m+1} \le C^m \, U^m,  \quad  \text{ where } \quad C^m := 1 + 4\lambda r_\phi \Phi_{\max}^m + L\phi(k).
\end{equation}
We now propagate this bound by induction to the next time level. 
At $m = 0$, the initial data are bounded by assumption, so $U^0 = \max_{i,j} u_{ij}^0 < \infty$.\\
Suppose $U^m < \infty$ for some $m \ge 0$ and let us show that $U^{m+1} < \infty$.
Note that $C^m$ depends only on $\lambda$, $r_\phi$, $\Phi_{\max}^m$, $L$, and $\phi(k)$, all of which are finite since $U^m < \infty$. 
Therefore, applying \eqref{E3.1} gives
\begin{equation*}
U^{m+1} \le C^m \, U^m < \infty.
\end{equation*}
Applying the recursive inequality~\eqref{E3.1} repeatedly 
from level $0$ to level $m$ gives
\begin{equation*}
   U^m \le \biggl(\prod_{\ell=0}^{m-1} C^\ell\biggr) U^0.
\end{equation*}
On a finite time interval $[0,T]$ with $m \le \lfloor T/k \rfloor$, 
each factor satisfies $C^\ell \le C$, where $C := 1 + 4\lambda r_\phi \Phi_{\max} + L\phi(k)$ and $\Phi_{\max} = \sup_m \Phi_{\max}^m$. 
Since the product contains exactly $m$ factors, we conclude
\begin{equation*}
U^m \le C^m \, U^0 \le C^{\lfloor T/k \rfloor} U^0 < \infty,
\end{equation*}
where $C^m$ denotes the $m$-th power of the uniform constant $C$.
This shows that the discrete $\ell^\infty$-norm remains bounded on $[0,T]$.
\end{proof}

\begin{remark}\label{R3}
The proof shows that the updated scheme~\eqref{E2.3} is not only positivity-preserving but also $\ell^\infty$-stable.
In particular, the incorporation of $f^-$ into the denominator prevents the reaction term from producing unbounded growth through the negative part, while the positive part is controlled by the Lipschitz bound on $f$.
\end{remark}

Next, we examine the consistency of~\eqref{E2.3}. 
Proposition~\ref{P1} shows that the scheme approximates~\eqref{E1.1} with first-order accuracy in time and second-order accuracy in space.

\begin{proposition}[Consistency]\label{P1}
The NSFD scheme~\eqref{E2.3} is consistent with the continuous problem~\eqref{E1.1}. 
In particular, its local truncation error satisfies $\tau_{ij}^m = \mathcal{O}(k + h^2), \, \forall (i,j)$.
\end{proposition}
\begin{proof}
Let $u$ be a sufficiently smooth solution of \eqref{E1.1}.
Substituting the exact solution into the FDM \eqref{E2.1} and measuring the residual defines
the local truncation error $\tau_{ij}^m$.\\
\textit{1. Temporal term.}
Since $\phi(k) = k + \mathcal{O}(k^2)$, a Taylor expansion in time gives
\begin{equation*}
      \frac{u(x_i,y_j,t_m+k) - u(x_i,y_j,t_m)}{\phi(k)}
   = \frac{k\,u_t(x_i,y_j,t_m) + \mathcal{O}(k^2)}{k + \mathcal{O}(k^2)}
   = u_t(x_i,y_j,t_m) + \mathcal{O}(k).
\end{equation*}
\textit{2. Spatial term.}
A Taylor expansion of each neighbor value in~\eqref{E2.2} gives
\begin{equation*}
  \mathcal{D}_h^{NS}  u(x_i,y_j,t_m) = \Delta_p\,  u(x_i,y_j,t_m) + \mathcal{O}(h^2),
\end{equation*}
since the centered five-point stencil approximates the $p$-Laplacian $\Delta_p$ 
to second order in $h$.\\
\textit{3. Reaction term.}
Since 
$u(x_i,y_j,t_m+k) =u(x_i,y_j,t_m)+ \mathcal{O}(k)$ 
for smooth solutions, we have
\begin{equation*}\begin{split}
   f^+\bigl(u(x_i,y_j,t_m)\bigr) - f^-\bigl(u(x_i,y_j,t_m+k)\bigr)
   &= f^+\bigl(u(x_i,y_j,t_m)\bigr) - f^-\bigl(u(x_i,y_j,t_m)\bigr) + \mathcal{O}(k)\\
   &= f\bigl(u(x_i,y_j,t_m)\bigr) + \mathcal{O}(k).
\end{split}\end{equation*}
%
Combining the above steps and comparing with~\eqref{E1.1}, we obtain
   $\tau_{ij}^m = \mathcal{O}(k + h^2)$,
which proves first-order accuracy in time and second-order accuracy in space.
\end{proof}

The standard explicit FDM~\eqref{E1.2} is structurally simpler but generally less robust. 
In particular, the reaction term is treated explicitly, and the coefficient associated with the central node may become negative for sufficiently large time steps. 
Consequently, the standard scheme may fail to preserve positivity and may produce spurious oscillations.

In contrast, the NSFD scheme~\eqref{E2.3} incorporates the reaction contribution through a nonlocal denominator, thereby better capturing the growth and decay mechanisms of the continuous model~\eqref{E1.1}. 
This feature plays a key role in ensuring the preservation of positivity and boundedness.

\section{Numerical Results}\label{S4}

We illustrate the theoretical properties through numerical experiments in two spatial dimensions. 
All simulations are performed on the unit square $\Omega = [0,1]^2$ with the logistic reaction term $f(u) = u(1-u)$, diffusion coefficient $\lambda = 0.1$, and initial condition $u_0(x,y) = 0.5 + 0.4\sin(\pi x)\sin(\pi y)$.
Homogeneous Neumann boundary conditions are imposed throughout.
We adopt the denominator function $\phi(k) = (e^{Lk}-1)L^{-1}$, as described in Section~\ref{S2.2} (see, e.g., \cite{Ehrhardt2013, Zinihi2025NSFD}), where $L\in(0,1)$ is a Lipschitz constant associated with the right-hand side of~\eqref{E1.1}.

The exponent $p > 1$ is the defining parameter of the $p$-Laplacian operator.
The linear diffusion case $p = 2$ serves as a natural baseline, while $p < 2$ and $p > 2$ correspond, respectively, to singular and degenerate nonlinear diffusion, two qualitatively distinct regimes with markedly different analytical properties~\cite{Lindqvist2019, Vazquez2006}.
To probe the full range of qualitative behaviors, we select two complementary families of values:
\begin{itemize}
\item \emph{Half-integer increments:} $p \in \{1.5,\, 2.5,\, 3.5,\, 4.5,\, 5.5\}$.
These values straddle the classical case $p = 2$ and increase the degree of nonlinearity in uniform steps of $0.5$, providing a systematic picture of how both schemes respond as diffusion becomes increasingly degenerate.
\item \emph{Integer values:} $p \in \{2,\, 3,\, 4,\, 5,\, 6\}$. 
These are the standard reference points in the $p$-Laplacian literature and include the linear diffusion case $p = 2$.
\end{itemize}
We emphasize that both selections are, in a certain sense, \emph{arbitrary}: no particular physical or analytical argument singles out these values over any other choice of $p > 1$.
At the same time, they are \emph{canonical} in that they are regularly spaced, cover both the singular and degenerate regimes, and include the classical linear case.
The qualitative conclusions drawn from these experiments are expected to hold for any $p > 1$.

We compare the standard explicit FDM~\eqref{E1.2} and the proposed NSFD scheme~\eqref{E2.3} with the spatial step size $h=0.025$
combined with time steps $k \in \{0.01,\, 0.05,\, 0.10\}$, and all 10 values of $p$ from both families described above.
The minimum and maximum values of the numerical solution at the final time $T = 1$ are reported in Table~\ref{Tab1}.
The table merges both $p$-families into a single display for ease of comparison.
For each parameter triple $(h, k, p)$, we report four quantities: the minimum and maximum of $u^m$ at $T=1$ for the FDM and for the NSFD scheme.
The \emph{FDM status} column records either 
\textit{'positive'}
when the scheme completes without incident, or
\textit{'blow-up $t = t^*$'},
when the solution first exceeds $\|u^m\|_\infty > 10^4$ or becomes non-finite at time $t = t^*$.
At that point, the time loop is halted, and the last valid state is recorded, making the instability visible and quantifiable rather than masking it with silent \textit{NaN} propagation.
The \textit{NSFD status} column consistently reads 
\textit{'positive'},
and the solutions remain bounded throughout, confirming the theoretical results computationally.

On the finest spatial grid, the FDM is unstable for all ten values of $p$ and all three time steps considered.
For the smallest time step $k = 0.01$, blow-up occurs as early as $t^* = 0.010$ for $p = 1.5$ and $t^* = 0.020$--$0.090$ for the remaining values.
The extrema of the FDM solution at the last recorded step reach values of order $10^2$--$10^3$ in magnitude for moderate~$k$, and order $10^3$--$10^4$ for $k = 0.10$, entirely inconsistent with the continuous solution, which remains in the interval $[0,1]$.
The NSFD scheme, by contrast, produces a positive and bounded solution in every single case.
Its minimum values remain strictly positive, and its maximum values, while sometimes transiently large (a feature of the nonlinear diffusion operator for large $p$), are consistent with the qualitative dynamics of the continuous problem.
\begin{table}[htb]
\centering
\setlength{\tabcolsep}{0.5cm}
\caption{%
Minimum and maximum values of the numerical solution at $T=1$
for mesh size $h=0.025$.
For each pair $(k,p)$ we report $\min u^m$ and $\max u^m$ for
both schemes.
The FDM status column records the first blow-up time
(when $\|u^m\|_\infty>10^4$ or the solution becomes non-finite);
the NSFD remains always positive. 
}\label{Tab1}
\resizebox{\textwidth}{!}{
\begin{tabular}{cc|cc|cc|c|c}
\hline
\multirow{2}{*}{$k$} & \multirow{2}{*}{$p$}
& \multicolumn{2}{c|}{FDM}
& \multicolumn{2}{c|}{NSFD}
& \multirow{2}{*}{FDM status}
& \multirow{2}{*}{NSFD status} \\
& & $\min u^m$ & $\max u^m$ & $\min u^m$ & $\max u^m$ & & \\
\hline\hline
\multirow{10}{*}{$0.01$}
& $1.5$ & $0.500$ & $0.900$ & $0.500$ & $0.900$ & \bad{blow-up $t=0.010$} & \good{positive} \\
& $2$   & $-671.5$ & $871.2$ & $0.580$ & $0.827$ & \bad{blow-up $t=0.090$} & \good{positive} \\
& $2.5$ & $-3246.9$ & $4124.9$ & $123.7$ & $3903.3$ & \bad{blow-up $t=0.040$} & \good{positive} \\
& $3$   & $-57.2$ & $67.9$ & $0.087$ & $4122.3$ & \bad{blow-up $t=0.030$} & \good{positive} \\
& $3.5$ & $-368.5$ & $437.5$ & $0.257$ & $31.1$ & \bad{blow-up $t=0.030$} & \good{positive} \\
& $4$   & $-4136.9$ & $4645.3$ & $0.100$ & $2758.7$ & \bad{blow-up $t=0.030$} & \good{positive} \\
& $4.5$ & $-0.678$ & $1.739$ & $0.078$ & $5041.3$ & \bad{blow-up $t=0.020$} & \good{positive} \\
& $5$   & $-0.917$ & $1.967$ & $0.397$ & $3.702$ & \bad{blow-up $t=0.020$} & \good{positive} \\
& $5.5$ & $-1.168$ & $2.207$ & $0.404$ & $15.5$ & \bad{blow-up $t=0.020$} & \good{positive} \\
& $6$   & $-1.437$ & $2.461$ & $0.431$ & $86.6$ & \bad{blow-up $t=0.020$} & \good{positive} \\
\hline
\multirow{10}{*}{$0.05$}
& $1.5$ & $0.500$ & $0.900$ & $0.500$ & $0.900$ & \bad{blow-up $t=0.050$} & \good{positive} \\
& $2$   & $-596.8$ & $898.5$ & $0.517$ & $0.880$ & \bad{blow-up $t=0.250$} & \good{positive} \\
& $2.5$ & $-201.0$ & $220.5$ & $0.913$ & $114.5$ & \bad{blow-up $t=0.150$} & \good{positive} \\
& $3$   & $-3778.9$ & $4878.1$ & $0.095$ & $994.1$ & \bad{blow-up $t=0.150$} & \good{positive} \\
& $3.5$ & $-3.213$ & $4.457$ & $0.121$ & $2341.3$ & \bad{blow-up $t=0.100$} & \good{positive} \\
& $4$   & $-4.361$ & $5.575$ & $0.576$ & $419.5$ & \bad{blow-up $t=0.100$} & \good{positive} \\
& $4.5$ & $-5.516$ & $6.693$ & $0.586$ & $47.4$ & \bad{blow-up $t=0.100$} & \good{positive} \\
& $5$   & $-6.709$ & $7.837$ & $0.128$ & $61.6$ & \bad{blow-up $t=0.100$} & \good{positive} \\
& $5.5$ & $-7.965$ & $9.033$ & $0.237$ & $68.7$ & \bad{blow-up $t=0.100$} & \good{positive} \\
& $6$   & $-9.312$ & $10.304$ & $0.375$ & $644.5$ & \bad{blow-up $t=0.100$} & \good{positive} \\
\hline
\multirow{10}{*}{$0.10$}
& $1.5$ & $0.500$ & $0.900$ & $0.500$ & $0.900$ & \bad{blow-up $t=0.100$} & \good{positive} \\
& $2$   & $-371.1$ & $233.1$ & $0.508$ & $0.890$ & \bad{blow-up $t=0.400$} & \good{positive} \\
& $2.5$ & $-1107.1$ & $1237.7$ & $0.216$ & $121.4$ & \bad{blow-up $t=0.300$} & \good{positive} \\
& $3$   & $-4.611$ & $6.574$ & $0.088$ & $2838.5$ & \bad{blow-up $t=0.200$} & \good{positive} \\
& $3.5$ & $-6.958$ & $8.413$ & $0.442$ & $2.710$ & \bad{blow-up $t=0.200$} & \good{positive} \\
& $4$   & $-9.254$ & $10.650$ & $0.428$ & $5596.9$ & \bad{blow-up $t=0.200$} & \good{positive} \\
& $4.5$ & $-11.564$ & $12.886$ & $0.587$ & $227.0$ & \bad{blow-up $t=0.200$} & \good{positive} \\
& $5$   & $-13.949$ & $15.175$ & $0.021$ & $496.3$ & \bad{blow-up $t=0.200$} & \good{positive} \\
& $5.5$ & $-16.462$ & $17.567$ & $0.257$ & $646.9$ & \bad{blow-up $t=0.200$} & \good{positive} \\
& $6$   & $-19.156$ & $20.108$ & $0.188$ & $909.1$ & \bad{blow-up $t=0.200$} & \good{positive} \\
\hline
\end{tabular}}
\end{table}

Coarsening the spatial grid from $h = 0.025$ to $h = 0.05$ does not prevent the FDM from blowing up; this occurs for all 10 $p$-values at every time step.
The blow-up times are generally larger than those at $h = 0.025$, reflecting the fact that the parabolic mesh ratio $r = k/h^2$ decreases as $h$ increases for a fixed~$k$, slightly relaxing the stability constraint.
Nevertheless, the FDM remains globally unstable.
For $k = 0.10$ the FDM blow-up times extend to $t^* = 0.600$ for $p = 2$, a relatively late failure, yet it still prevents the scheme from reaching $T = 1$.
The NSFD scheme remains positive and bounded in all cases.

The NSFD scheme is positive and bounded in all 30 cases without any time-step restriction, while standard FDM is always unstable.
These results provide compelling computational evidence for the unconditional positivity stated in Theorem~\ref{T1} and the boundedness established in Theorem~\ref{T3}.
For completeness, Figures~\ref{Fig1}--\ref{Fig6} are presented in Appendix~\ref{App1}, showing the two-dimensional solution at the final time $T=1$ for all parameter combinations.

\section*{Conclusion and Future Work}\label{S5}

In this work, we have proposed and analyzed an NSFD scheme for a class of nonlinear parabolic equations involving the $p$-Laplacian diffusion operator in one- and two-dimensional spatial domains.
The theoretical contributions of this paper are threefold.
First, we proved that the NSFD scheme~\eqref{E2.3} preserves positivity unconditionally (Theorem~\ref{T1}): if the solution is positive at one time level, then it remains positive at the next time level for any time step size $k > 0$ and any spatial mesh $h > 0$.
Second, we demonstrated that every constant equilibrium $u^*$ of the continuous problem is reproduced exactly by the scheme (Theorem~\ref{T2}), a property that standard explicit methods generally fail to maintain on coarse meshes.
Third, we derived an a priori $\ell^\infty$ bound on the discrete solution (Theorem~\ref{T3}) and proved that the local truncation error satisfies $\tau_{ij}^m = \mathcal{O}(k + h^2)$ (Proposition~\ref{P1}), confirming first-order accuracy in time and second-order accuracy in space.

The numerical experiments corroborate all three results across a comprehensive set of parameter combinations, covering two families of exponents ($p \in \{1.5, 2.5, 3.5, 4.5, 5.5\}$ and $p \in \{2, 3, 4, 5, 6\}$), 
and three time steps ($k \in \{0.01, 0.05, 0.10\}$).
Out of the 30 parameter combinations tested, the standard explicit FDM produced nonphysical negative values and blew up within a small number of time steps in all cases, while the NSFD scheme remained positive, bounded, and convergent.


The proposed NSFD scheme has several practical advantages.
It is fully explicit, requiring no iterative solver at each time step. 
Thus, its computational cost per step is comparable to that of standard FDM.
The positivity and boundedness guarantees eliminate the need for ad hoc clipping or postprocessing of the numerical solution.
The NSFD scheme continues to produce qualitatively correct solutions on coarse temporal grids where the standard FDM is unusable. 
This makes the NSFD scheme particularly attractive for long-time simulations or parameter studies where large time steps are necessary for economic reasons.\\
The results of this paper open several natural lines of investigation.
\begin{enumerate}
\item \textbf{Higher-Order NSFD Schemes.} 
Building on the generalized second-order framework introduced in~\cite{hoang2026generalized}, it would be interesting to construct a second-order-in-time positivity-preserving NSFD scheme for the $p$-Laplacian equation, overcoming the order reduction identified as a limitation above.
\item \textbf{Fractional and Nonlocal Diffusion.} 
A natural extension is to replace the $p$-Laplacian operator with a fractional $p$-Laplacian or a nonlocal diffusion kernel, as in anomalous diffusion models and peridynamics~\cite{Vazquez2006}.
Designing structure-preserving NSFD schemes in this setting presents new analytical and computational challenges.
\item \textbf{Reaction-Diffusion Systems.} 
The epidemic and actuarial models studied in~\cite{Zinihi2026Actuarial, Zinihi2025FDE} involve systems of coupled parabolic equations with $p$-Laplacian-type diffusion.
An important direction for applications in mathematical epidemiology is extending the present NSFD framework to such systems while preserving the positivity of each component and the conservation structure of the system.
\item \textbf{Machine Learning Integration.} 
Recent work~\cite{Zinihi2026CA} has demonstrated that NSFD schemes can generate structure-preserving synthetic data for training physics-informed neural networks (PINNs).
Applying this paradigm to the $p$-Laplacian equation offers a promising approach to data-driven surrogate modeling that retains the qualitative guarantees of the underlying numerical scheme.
\end{enumerate}
We hope that the framework developed here will serve as a foundation for structure-preserving numerical methods for a broader class of nonlinear degenerate parabolic equations arising in science and engineering.

\begin{appendices}
\section{Additional Numerical Illustrations}\label{App1}

Figures~\ref{Fig1}--\ref{Fig6} display the two-dimensional solution at the 
final time $T=1$ for all parameter combinations.
The figures are organized as follows:
\begin{itemize}
\item Figures~\ref{Fig1}, \ref{Fig3}, and \ref{Fig5} correspond to the half-integer family $p \in \{1.5, 2.5, 3.5, 4.5, 5.5\}$ for $h \in \{0.025, 0.05, 0.1\}$, respectively.
\item Figures~\ref{Fig2}, \ref{Fig4}, and \ref{Fig6} correspond to the integer family $p \in \{2, 3, 4, 5, 6\}$ for the same sequence of mesh sizes.
\end{itemize}
Each figure contains five rows (one for each value of $p$) and six columns, which are arranged in three column-pairs, with one pair per time step $k \in \{0.01, 0.05, 0.10\}$. Within each pair, the left panel shows the FDM solution and the right panel shows the NSFD solution.

The FDM panels showing blow-up are displayed with a diverging red-blue color map centered at zero, to make spurious sign changes visible.
The title of each blow-up panel indicates the blow-up time $t^*$, and the plotted field corresponds to the last valid state before blow-up.
The NSFD panels uniformly use the viridis color map, indicating that the solution remains positive throughout.

Several physical features are clearly visible in the NSFD panels.
For small $p$ (close to or below $2$), the diffusion operator is singular, and the solution profile retains the spatial structure inherited from the initial condition. It slowly converges toward the uniform steady state $u^* = 1$.
For large $p$, the diffusion is strongly degenerate, producing flatter interior profiles with steeper gradients concentrated near the domain boundary, a well-known qualitative feature of the $p$-Laplacian equation~\cite{Lindqvist2019}.
However, none of these features are visible in the FDM solutions, which are dominated by numerical artifacts.

\section*{Declarations}

\subsection*{Conflict of Interest} 
\noindent
The authors declared that they have no conflict of interest.

\subsection*{Data Availability} 
\noindent
No data was used for the research described in the article. 

\subsection*{Author Contributions}
\noindent
A. Zinihi: Conceptualization, Soft\-ware, Methodology, Validation, Formal Analysis, Investigation, Writing-Original Draft, Writing-Review and Editing, Visualization.

\noindent
M. Ehrhardt: Conceptualization, Supervision, Methodology, Formal Analysis, Investigation, Writing-Review and Editing.

\noindent
M. R. Sidi Ammi: Conceptualization, Supervision, Methodology, Formal Analysis.

\begin{figure}[H]
\centering 
\includegraphics[width=\textwidth, height=.56\textwidth]{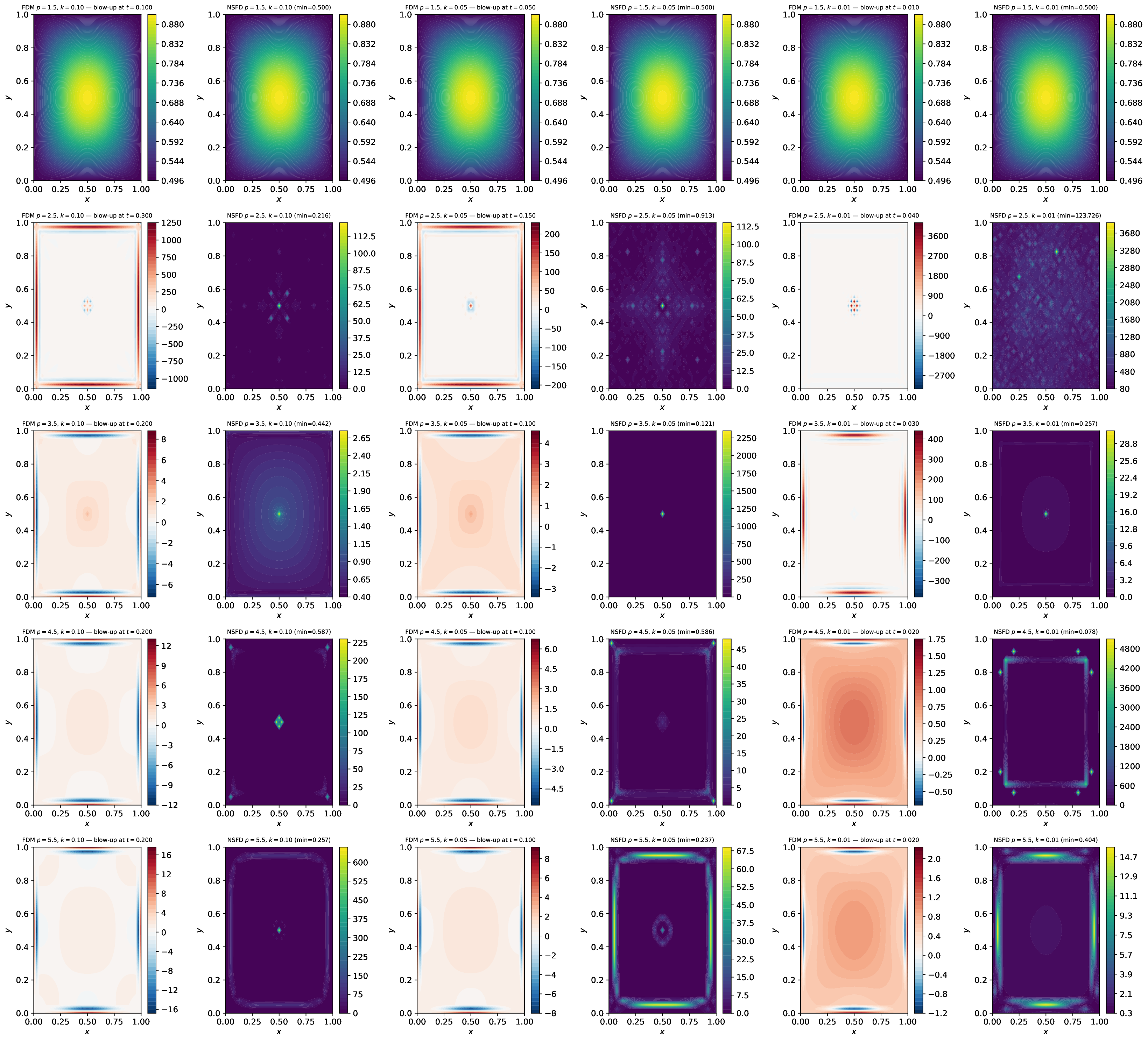}
\caption{Two-dimensional numerical solution at $T=1$ for the integer exponents $p \in \{1.5, 2.5, 3.5, 4.5, 5.5\}$ (rows) and time steps $k \in \{0.01, 0.05, 0.10\}$ (column pairs) with mesh size $h = 0.025$.}\label{Fig1}
\end{figure}

\begin{figure}[H]
\centering
\includegraphics[width=\textwidth, height=.56\textwidth]{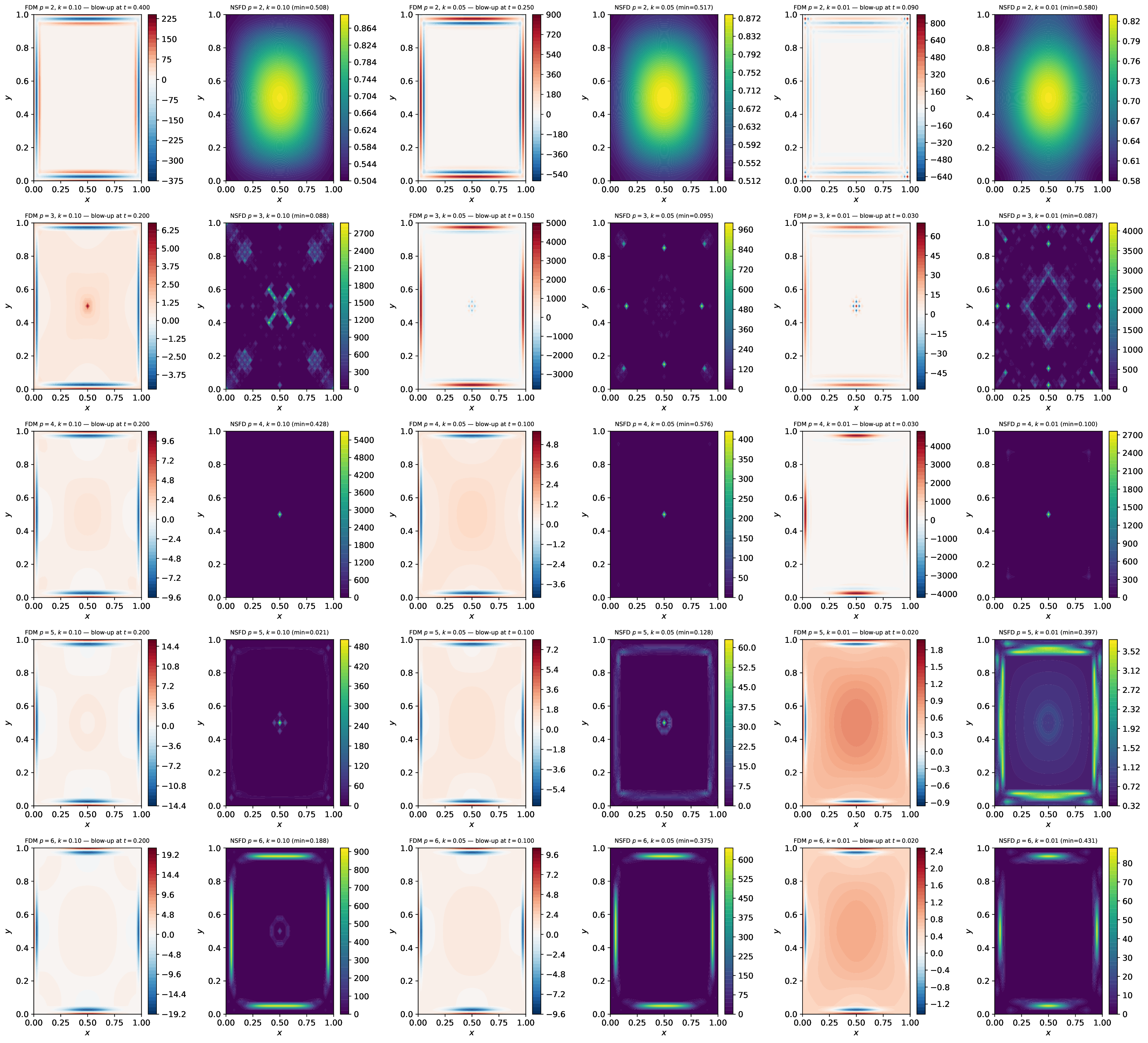}
\caption{Two-dimensional numerical solution at $T=1$ for the integer exponents $p \in \{2, 3, 4, 5, 6\}$ (rows) and time steps $k \in \{0.01, 0.05, 0.10\}$ (column pairs) with mesh size $h = 0.025$.}\label{Fig2}
\end{figure}

\begin{figure}[H]
\centering
\includegraphics[width=\textwidth, height=.56\textwidth]{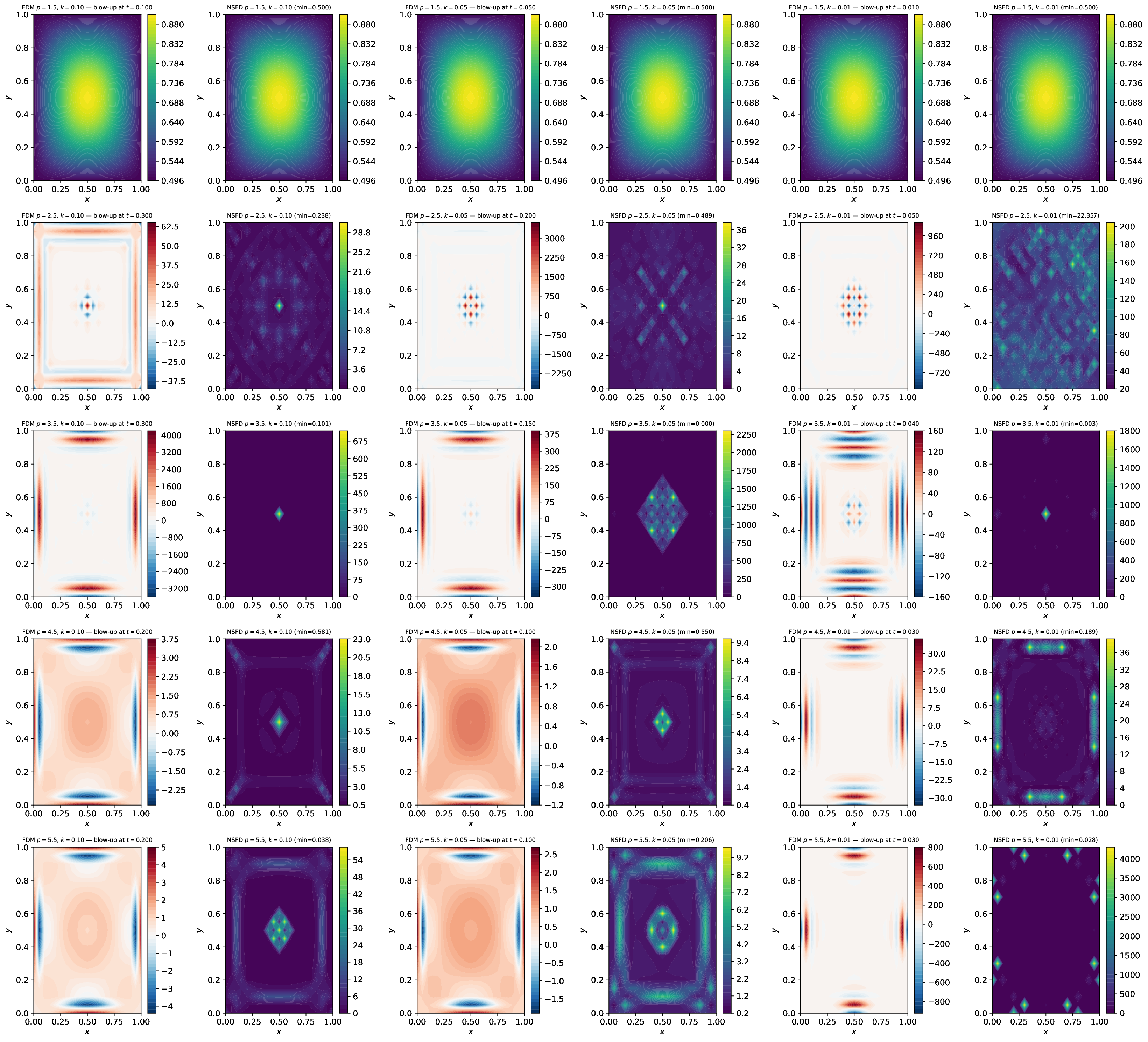}
\caption{Two-dimensional numerical solution at $T=1$ for the integer exponents $p \in \{1.5, 2.5, 3.5, 4.5, 5.5\}$ (rows) and time steps $k \in \{0.01, 0.05, 0.10\}$ (column pairs) with mesh size $h = 0.05$.}\label{Fig3}
\end{figure}

\begin{figure}[H]
\centering
\includegraphics[width=\textwidth, height=.56\textwidth]{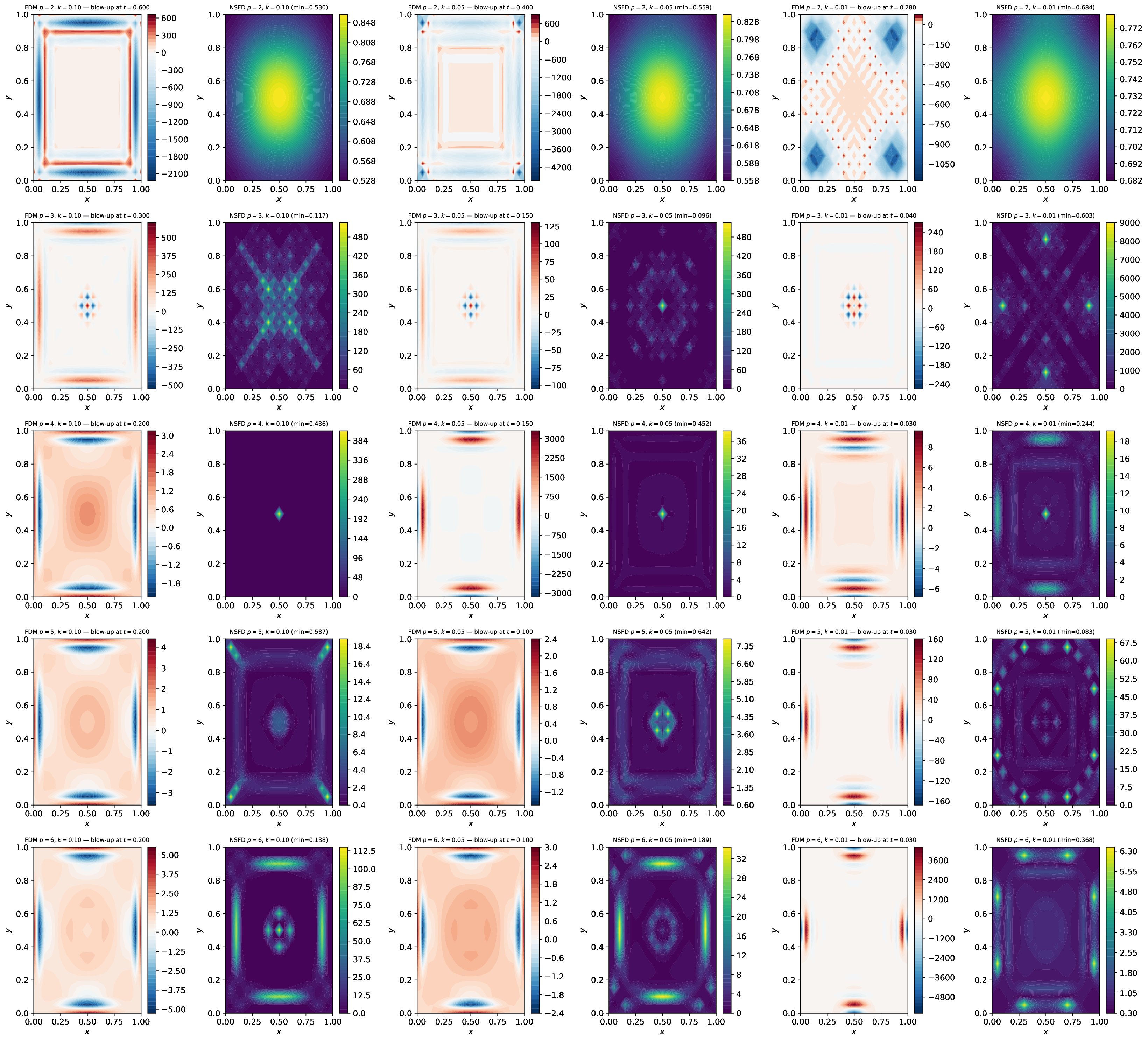}
\caption{Two-dimensional numerical solution at $T=1$ for the integer exponents $p \in \{2, 3, 4, 5, 6\}$ (rows) and time steps $k \in \{0.01, 0.05, 0.10\}$ (column pairs) with mesh size $h = 0.05$.}\label{Fig4}
\end{figure}

\begin{figure}[H]
\centering
\includegraphics[width=\textwidth, height=.56\textwidth]{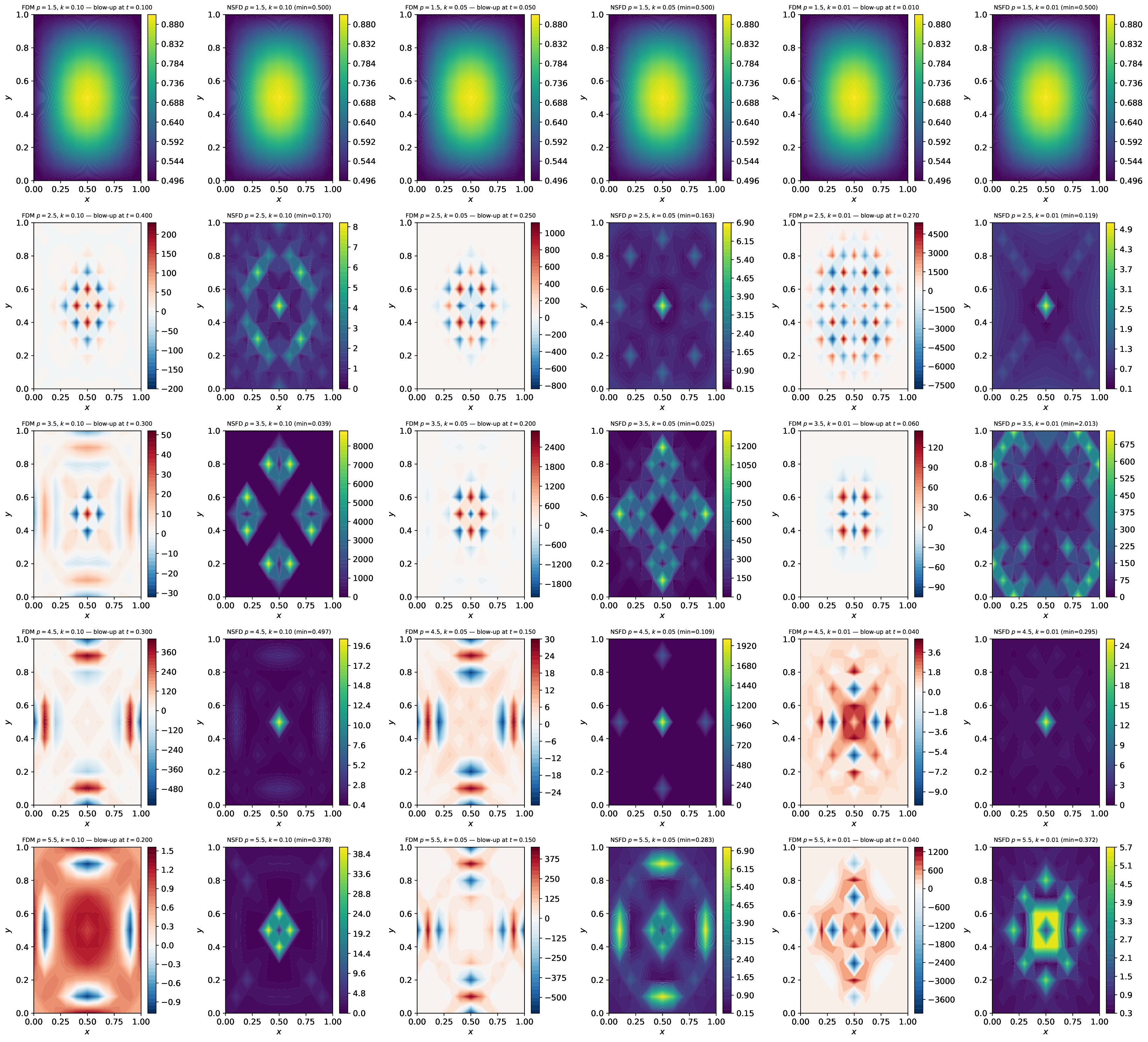}
\caption{Two-dimensional numerical solution at $T=1$ for the integer exponents $p \in \{1.5, 2.5, 3.5, 4.5, 5.5\}$ (rows) and time steps $k \in \{0.01, 0.05, 0.10\}$ (column pairs) with mesh size $h = 0.1$.}\label{Fig5}
\end{figure}

\begin{figure}[H]
\centering 
\includegraphics[width=\textwidth, height=.56\textwidth]{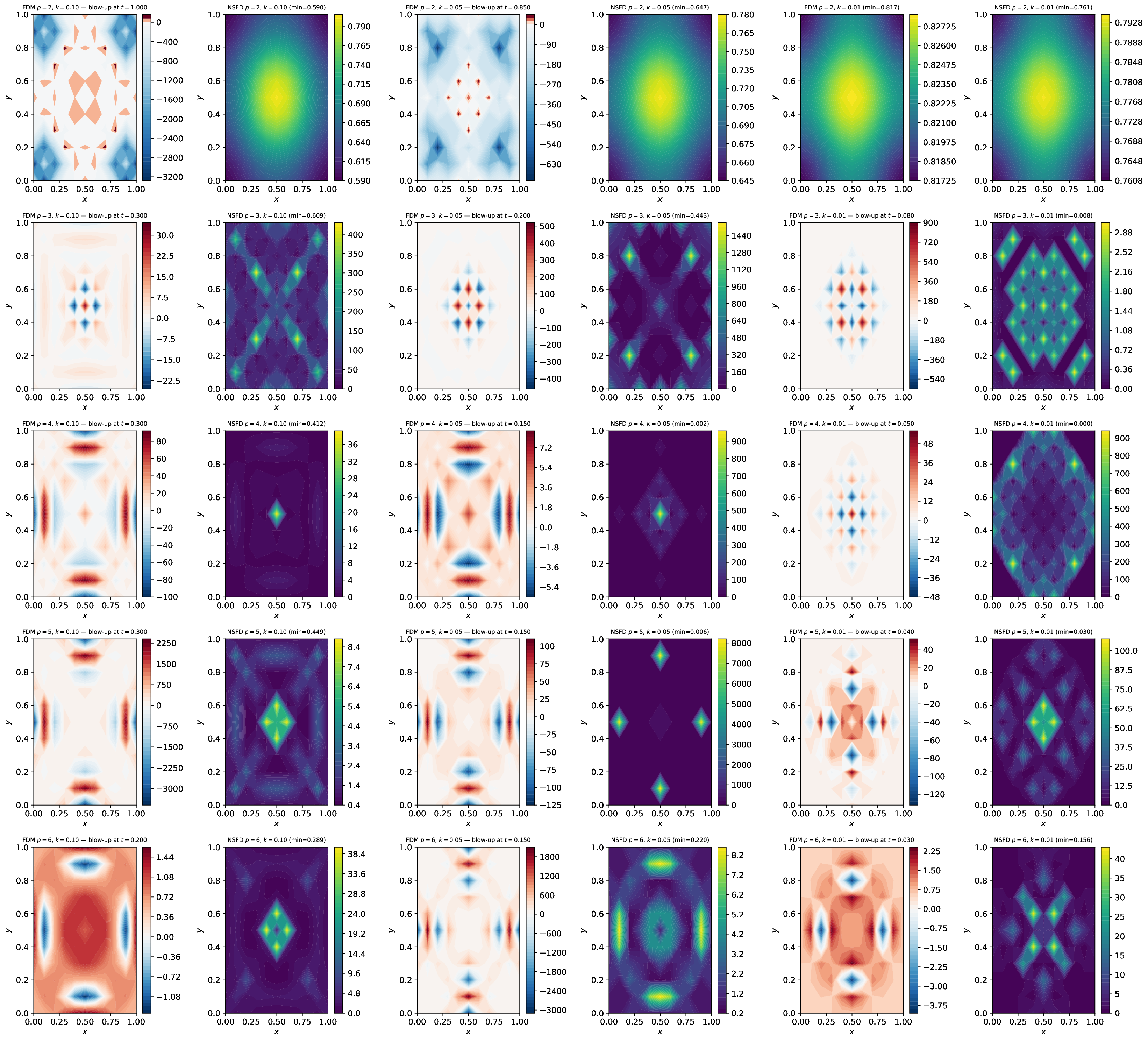}
\caption{Two-dimensional numerical solution at $T=1$ for the integer exponents $p \in \{2, 3, 4, 5, 6\}$ (rows) and time steps $k \in \{0.01, 0.05, 0.10\}$ (column pairs) with mesh size $h = 0.1$.}\label{Fig6}
\end{figure}
\end{appendices}



\bibliographystyle{unsrt}
\bibliography{References}


\end{document}